\documentclass[12pt]{amsart}
\usepackage{amsmath,amsrefs,amsthm,amssymb}
\usepackage{tikz-cd}
\usepackage{enumitem}
\usepackage[capitalize]{cleveref}
\usepackage{newtxtext}
\usepackage{mathtools}

\newcommand{\C}{\mathcal{C}}
\newcommand{\G}{\mathcal{G}}
\newcommand{\I}{\mathcal{I}}
\renewcommand{\S}{\mathcal{S}}
\renewcommand{\P}{\mathcal{P}}

\newcommand{\Ker}[1]{\operatorname{Ker}(#1)}

\newcommand{\E}{\mathbb{E}}
\newcommand{\s}{\mathfrak{s}}
\newcommand{\F}{\mathcal{F}}
\newcommand{\R}{\mathcal{R}}
\newcommand{\Ab}{Ab}
\newcommand{\widebar}[1]{\overline{#1}}

\makeatletter
\newcommand{\xrightleftarrows}[2]{%
  \mathrel{\mathop{%
    \vcenter{\offinterlineskip\m@th
      \ialign{\hfil##\hfil\cr
        \hphantom{$\scriptstyle\mspace{8mu}{#1}\mspace{8mu}$}\cr
        \rightarrowfill\cr
        \vrule height0pt width 2em\cr
        \leftarrowfill\cr
        \hphantom{$\scriptstyle\mspace{8mu}{#2}\mspace{8mu}$}\cr
        \noalign{\kern-0.3ex}
      }%
    }%
  }\limits^{#1}_{#2}}%
}
\makeatother

\newtheorem{theorem}{Theorem}[section]
\newtheorem{lemma}[theorem]{Lemma}
\newtheorem{corollary}[theorem]{Corollary}
\newtheorem{proposition}[theorem]{Proposition}

\theoremstyle{definition}
\newtheorem*{acknowledgements}{Acknowledgements}
\newtheorem{definition}[theorem]{Definition}
\newtheorem{remark}[theorem]{Remark}
\newtheorem{example}[theorem]{Example}

\title{The Grothendieck group of an $n$-exangulated category}

\author{Johanne Haugland}

\begin{document}

\keywords{Grothendieck group, $n$-exangulated category, $(n+2)$-angulated category, $n$-exact category, $n$-exangulated subcategory, extriangulated subcategory}
\subjclass[2010]{18E10, 18E30, 18F30} 

\address{Department of mathematical sciences, NTNU, NO-7491 Trondheim, Norway}
\email{johanne.haugland@ntnu.no}

\begin{abstract}
We define the Grothendieck group of an $n$-exangulated category. For $n$ odd, we show that this group shares many properties with the Grothendieck group of an exact or a triangulated category. In particular, we classify dense complete subcategories of an $n$-exangulated category with an $n$-(co)generator in terms of subgroups of the Grothendieck group. This unifies and extends results of Thomason, Bergh--Thaule, Matsui and Zhu--Zhuang for triangulated, $(n+2)$-angulated, exact  and extriangulated categories, respectively. We also introduce the notion of an $n$-exangulated subcategory and prove that the subcategories in our classification theorem carry this structure.
\end{abstract}

\maketitle

\section{Introduction}

The Grothendieck group of an exact category is the free abelian group generated by isomorphism classes of objects modulo the Euler relations coming from short exact sequences. Similarly, one obtains the Grothendieck group of a triangulated category by factoring out the relations corresponding to distinguished triangles. It turns out that subcategories of certain categories relate to subgroups of the associated Grothendieck groups in an interesting way. More precisely, Thomason proved that there is a one-to-one correspondence between subgroups of the Grothendieck group of a triangulated category and dense triangulated subcategories \cite{Thomason}*{Theorem 2.1}. This was generalized to $(n+2)$-angulated categories with $n$ odd by Bergh--Thaule \cite{Bergh-Thaule}*{Theorem 4.6}. Later, Matsui gave an analogous result for exact categories with a (co)generator \cite{Matsui}*{Theorem 2.7}. 

The notion of extriangulated categories was introduced by Nakaoka--Palu as a simultaneous generalization of exact categories and triangulated categories \cite{Nakaoka-Palu}. Many concepts and results concerning exact and triangulated categories have been unified and extended using this framework, see for instance \cite{Iyama-Nakaoka-Palu} for a generalization of Auslander--Reiten theory in exact and triangulated categories to this context.

In both higher dimensional Auslander--Reiten theory and higher homological algebra, $n$-cluster tilting subcategories of exact and triangulated categories play a fundamental role \cites{Iyama,Iyama-Oppermann}. This was a starting point for developing the theory of $(n+2)$-angulated categories and $n$-exact categories in the sense of Geiss--Keller--Oppermann \cite{Geiss-Keller-Oppermann} and Jasso \cite{Jasso}. Recently, Herschend--Liu--Nakaoka defined $n$-exangulated categories as a higher dimensional analogue of extriangulated categories \cite{Herschend-Liu-Nakaoka}. Many categories studied in representation theory turn out to be $n$-exangulated. In particular, $n$-exangulated categories simultaneously generalize $(n+2)$-angulated and $n$-exact categories. In \cite{Herschend-Liu-Nakaoka}*{Section 6} several explicit examples of $n$-exangulated categories are given. See also \cite{Liu-Zhou}*{Section 4} for a construction which yields more $n$-exangulated categories that are neither $n$-exact nor $(n+2)$-angulated.

Inspired by the classification results for triangulated, \mbox{$(n+2)$}-angulated and exact categories mentioned above, a natural question to ask is whether there is a similar connection between subcategories and subgroups of the Grothendieck group for $n$-exangulated categories. Independently of our work, Zhu--Zhuang recently gave a partial answer to this question in the case $n=1$ \cite{Zhu-Zhuang}*{Theorem 5.7}. In this paper we prove a more general classification result for $n$-exangulated categories with $n$ odd. In our main result, \cref{main theorem}, we classify dense complete subcategories of an $n$-exangulated category with an $n$-(co)generator $\G$ in terms of subgroups of the Grothendieck group containing the image of $\G$. This recovers both the result of Zhu--Zhuang for extriangulated categories and the result of Bergh--Thaule for $(n+2)$-angulated categories, as well as Thomason's and Matsui's results for triangulated and exact categories, see \cref{corollary} for details. Our main theorem also yields new classification results for $n$-exact categories, as well as for $n$-exangulated categories which are neither $(n+2)$-angulated nor $n$-exact.

The paper is organized as follows. In Section 2 we recall the definition of an $n$-exangulated category and review some results. In Section 3 we explain terminology which is needed in our main result, such as the notion of an $n$-(co)generator, complete subcategories and dense subcategories. We also introduce $n$-exangulated subcategories and prove that the subcategories which will appear in our classification theorem carry this structure. In Section 4 we define the Grothendieck group of an $n$-exangulated category and discuss some basic results. In Section 5 we state and prove our main theorem and explain how this unifies and extends already known results. 

\section{Preliminaries on $n$-exangulated categories}

Throughout this paper, let $n$ be a positive integer and $\C$ an additive category. In this section we briefly recall the definition of an $n$-exangulated category and related notions, as well as some known results which will be used later. All of this is taken from \cite{Herschend-Liu-Nakaoka}, and we recommend to consult this paper for more detailed explanations. 

Recall from \cite{Nakaoka-Palu} that an extriangulated category $(\C,\E,\s)$ consists of an additive category $\C$, a biadditive functor $\E \colon \C^{\operatorname{op}} \times \C \rightarrow \Ab$ and an additive realization $\s$ of $\E$ satisfying certain axioms. The functor $\E$ is modelled after $\operatorname{Ext}^1$. Given two objects $A$ and $C$ in $\C$, the realization $\s$ associates to each element $\delta \in \E(C,A)$ an equivalence class $\s (\delta)$ of $3$-term sequences in $\C$ starting in $A$ and ending in $C$. Exact and triangulated categories are examples of extriangulated categories, where short exact sequences and distinguished triangles play the roles of these $3$-term sequences. Analogously, an $n$-exangulated category also consists of a triplet $(\C,\E,\s)$, where the main difference is that we consider $(n+2)$-term sequences instead of $3$-term sequences. In order to give the precise definition, we need to be able to talk about extensions and morphisms of extensions.

\begin{definition}
Let $\E \colon \C^{\operatorname{op}} \times \C \rightarrow \Ab$ be a biadditive functor. Given two objects $A$ and $C$ in $\C$, an element $\delta \in \E (C,A)$ is called an \textit{$\E$-extension} or simply an \textit{extension}. We can write such an extension $\delta$ as ${}_A \delta_C$ whenever we wish to specify the objects $A$ and $C$.
\end{definition}

Given an extension $\delta \in \E(C,A)$ and two morphisms $a \in \C (A,A')$ and \mbox{$c \in \C(C',C)$}, we denote the extensions
\[
\E(C,a)(\delta) \in \E(C,A') ~\text{ and }~ \E(c,A)(\delta) \in \E(C',A)
\]
by $a_* \delta$ and $c^* \delta$. Notice that $\E(c,a)(\delta)=c^*a_* \delta = a_* c^* \delta$ in $\E(C',A')$ as $\E$ is a bifunctor.

For any pair of objects $A$ and $C$, the zero element ${}_A 0_C$ in $\E(C,A)$ is called the \textit{split extension}. 

\begin{definition}
Given extensions ${}_A \delta_C$ and ${}_B \rho_D$, a \textit{morphism of extensions} $(a,c) \colon \delta \rightarrow \rho$ is a pair of morphisms $a \in \C(A,B)$ and $c \in \C(C,D)$ such that $a_* \delta = c^* \rho$ in $\E(C,B)$.
\end{definition}

We want to associate each extension ${}_A \delta_C$ to an equivalence class of $(n+2)$-term sequences in $\C$ starting in $A$ and ending in $C$. Our next aim is hence to discuss some terminology which will enable us to describe the appropriate equivalence relation on the class of such $(n+2)$-term sequences.

\begin{definition}
Let $\textbf{C}_\C$ denote the category of complexes in $\C$. We define $\textbf{C}_\C^{n+2}$ to be the full subcategory of $\textbf{C}_\C$ consisting of complexes whose components are zero in all degrees outside of $\{0,1,\dots,n+1\}$. In other words, an object in $\textbf{C}_\C^{n+2}$ is a complex $X_\bullet = \{X_i,d_i\}$ of the form
\[
X_0 \xrightarrow{d_0} X_1 \rightarrow \cdots \rightarrow X_n \xrightarrow{d_n} X_{n+1}.
\]
Morphisms in $\textbf{C}_\C^{n+2}$ are written $f_\bullet = (f_0,f_1,\dots,f_{n+1})$, where we only indicate the terms of degree $0,1,\dots,n+1$.
\end{definition}

Our next two definitions should remind the reader about the long exact $\operatorname{Hom}$-$\operatorname{Ext}$-sequence associated to a short exact sequence and the long exact $\operatorname{Hom}$-sequence associated to a distinguished triangle.

\begin{definition}
By the Yoneda lemma, an extension $\delta \in \E(C,A)$ induces natural transformations
\[
\delta_\sharp \colon \C(-,C) \rightarrow \E(-,A) ~\text{ and }~ \delta^\sharp \colon \C(A,-) \rightarrow \E(C,-).
\]
For an object $X$ in $\C$, the morphisms $(\delta_\sharp)_X$ and $\delta^\sharp_X$ are given by
    \begin{enumerate}
        \item $(\delta_\sharp)_X \colon \C(X,C) \rightarrow \E(X,A)$, $f \mapsto f^*\delta$;
        \item $\delta^\sharp_X \colon \C(A,X) \rightarrow \E(C,X)$, $g \mapsto g_*\delta$.
    \end{enumerate}
\end{definition}

Consider a pair $\langle X_\bullet,\delta \rangle$ with $X_\bullet$ in $\textbf{C}_\C^{n+2}$ and $\delta \in \E(X_{n+1},X_0)$. Using our natural transformations from above, we can associate to $\langle X_\bullet,\delta \rangle$ the following two sequences of functors:
\begin{enumerate}
    \item $\C(-,X_0) \xrightarrow{\mathmakebox[2.5em]{\C(-,d_0)}} \cdots \xrightarrow{\mathmakebox[2.5em]{\C(-,d_n)}} \C(-,X_{n+1}) \xrightarrow{\mathmakebox[2.5em]{\delta_\sharp}} \E(-,X_0)$;
    \item $\C(X_{n+1},-) \xrightarrow{\mathmakebox[2.5em]{\C(d_n,-)}} \cdots \xrightarrow{\mathmakebox[2.5em]{\C(d_0,-)}} \C(X_0,-) \xrightarrow{\mathmakebox[2.5em]{\delta^\sharp}} \E(X_{n+1},-)$.
\end{enumerate}
We are particularly interested in pairs $\langle X_\bullet,\delta \rangle$ for which these sequences are exact.

\begin{definition}
When the two sequences of functors from above are exact, we say that the pair $\langle X_\bullet,\delta \rangle$ is an \textit{$n$-exangle}. Given two $n$-exangles $\langle X_\bullet,\delta \rangle$ and $\langle Y_\bullet,\rho \rangle$, a \textit{morphism of $n$-exangles} $f_\bullet \colon \langle X_\bullet,\delta \rangle \rightarrow \langle Y_\bullet,\rho \rangle$ is a chain map \mbox{$f_\bullet \in \textbf{C}_\C^{n+2}(X_\bullet,Y_\bullet)$} for which $(f_0,f_{n+1}) \colon \delta \rightarrow \rho$ is also a morphism of extensions.
\end{definition}

In order to define our equivalence classes of $(n+2)$-term sequences, we need a notion of homotopy. Two morphisms in $\textbf{C}_\C^{n+2}$ are said to be \textit{homotopic} if they are homotopic as morphisms of $\textbf{C}_\C$ in the usual way. We let the homotopy category $\textbf{K}_\C^{n+2}$ be the quotient of $\textbf{C}_\C^{n+2}$ by the ideal of null-homotopic morphisms. 

Instead of working with $\textbf{C}_\C^{n+2}$ and $\textbf{K}_\C^{n+2}$, we want to fix the end-terms of our sequences. 

\begin{definition}
Let $A$ and $C$ be objects in $\C$. We define $\textbf{C}_{(\C;A,C)}^{n+2}$ to be the subcategory of $\textbf{C}_\C^{n+2}$ consisting of complexes $X_\bullet$ with $X_0 = A$ and $X_{n+1}=C$. Morphisms in $\textbf{C}_{(\C;A,C)}^{n+2}$ are given by chain maps $f_\bullet$ for which $f_0 = 1_A$ and $f_{n+1}=1_C$.
\end{definition}

Whenever the category $\C$ is clear from the context, we abbreviately denote $\textbf{C}_{(\C;A,C)}^{n+2}$ by $\textbf{C}_{(A,C)}^{n+2}$.
Notice that $\textbf{C}_{(A,C)}^{n+2}$ is no longer an additive category. However, we can still take the quotient of $\textbf{C}_{(A,C)}^{n+2}$ by the same homotopy relation as in $\textbf{C}_\C^{n+2}$. This yields $\textbf{K}_{(A,C)}^{n+2}$, which is a subcategory of $\textbf{K}_\C^{n+2}$.

We are now ready to describe an equivalence relation on the class of $(n+2)$-term sequences starting in $A$ and ending in $C$.

\begin{definition}
A morphism $f_\bullet \in \textbf{C}_{(A,C)}^{n+2}(X_\bullet,Y_\bullet)$ is called a \textit{homotopy equivalence} if it induces an isomorphism in $\textbf{K}_{(A,C)}^{n+2}$. Two objects $X_\bullet$ and $Y_\bullet$ in $\textbf{C}_{(A,C)}^{n+2}$ are called \textit{homotopy equivalent} if there is some homotopy equivalence between them. We denote the homotopy equivalence class of $X_\bullet$ by $[X_\bullet]$.
\end{definition}

It should be noted that homotopy equivalence classes taken in $\textbf{C}_{(A,C)}^{n+2}$ and in $\textbf{C}_\C^{n+2}$ may be different. We will only use the notation $[X_\bullet]$ for equivalence classes taken in $\textbf{C}_{(A,C)}^{n+2}$.

We are now ready to explain our desired connection between extensions ${}_A \delta_C$ and equivalence classes $[X_\bullet]$ in $\textbf{C}_{(A,C)}^{n+2}$.

\begin{definition}
Let $\s$ be a correspondence which associates a homotopy equivalence class $\s(\delta)=[X_\bullet]$ in $\textbf{C}_{(A,C)}^{n+2}$ to each extension $\delta \in \E(C,A)$. We call $\s$ a \textit{realization of $\E$} if it satisfies the following condition for any $\s(\delta)=[X_\bullet]$ and $\s(\rho)=[Y_\bullet]$:
\begin{enumerate}
    \item[(R0)] Given any morphism of extensions $(a,c) \colon \delta \rightarrow \rho$, there exists a morphism $f_\bullet \in \textbf{C}_\C^{n+2}(X_\bullet,Y_\bullet)$ of the form $f_\bullet = (a,f_1,\dots,f_n,c)$. Such an $f_\bullet$ is called a \textit{lift} of $(a,c)$.
\end{enumerate}
Whenever $\s(\delta)=[X_\bullet]$, we say that $X_\bullet$ \textit{realizes} $\delta$. A realization $\s$ is called \textit{exact} if in addition the following conditions hold:
\begin{enumerate}
    \item[(R1)] Given any $\s(\delta)=[X_\bullet]$, the pair $\langle X_\bullet,\delta \rangle$ is an $n$-exangle.
    \item[(R2)] Given any object $A$ in $\C$, we have \[\s({}_A 0_0) = [
A \xrightarrow{1_A} A \rightarrow 0 \rightarrow \cdots \rightarrow 0 
],\]
and dually
\[\s({}_0 0_A) = [0 \rightarrow \cdots \rightarrow 0 \rightarrow A \xrightarrow{1_A} A ].\]
\end{enumerate}
\end{definition}

It is not immediately clear that the condition (R1) does not depend on our choice of representative of the class $[X_\bullet]$. For this fact, see \cite{Herschend-Liu-Nakaoka}*{Proposition 2.16}.

Based on the definition above, we can introduce some useful terminology.

\begin{definition}
Let $\s$ be an exact realization of $\E$.
\begin{enumerate}
    \item An $n$-exangle $\langle X_\bullet,\delta \rangle$ will be called a \textit{distinguished $n$-exangle} if $\s(\delta) = [X_\bullet]$. 
    \item An object $X_\bullet \in \textbf{C}_\C^{n+2}$ will be called a \textit{conflation} if it realizes some extension $\delta \in \E(X_{n+1},X_0)$.
    \item A morphism $f$ in $\C$ will be called an \textit{inflation} if there exists some conflation $X_\bullet = \{X_i,d_i\}$ satisfying $d_0 = f$.
    \item A morphism $g$ in $\C$ will be called a \textit{deflation} if there exists some conflation $X_\bullet = \{X_i,d_i\}$ satisfying $d_n = g$.
\end{enumerate}
\end{definition}

Recall that for triangulated categories, the octahedral axiom can be replaced by a mapping cone axiom \cites{Neeman 1991,Neeman 2001}. This should be thought of as a background for the definition of an $n$-exangulated category. Before we can give the definition, we need the notion of a mapping cone in our context. 

\begin{definition}
Let $f_\bullet \in \textbf{C}_\C^{n+2}(X_\bullet,Y_\bullet)$ be a morphism with $f_0 = 1_A$ for some object $A=X_0=Y_0$ in $\C$. The \textit{mapping cone} of $f_\bullet$ is the complex $M^f_\bullet \in \textbf{C}_\C^{n+2}$ given by
\[
X_1 \xrightarrow{d_0} X_2 \oplus Y_1 \xrightarrow{d_1} X_3 \oplus Y_2 \xrightarrow{d_2} \cdots \xrightarrow{d_{n-1}} X_{n+1} \oplus Y_n \xrightarrow{d_n} Y_{n+1},
\]
where
\[
d_i = 
\begin{cases}
        \vspace{-1em}
        \begin{bmatrix}
         -d_1^X \\
        f_1 
        \end{bmatrix} & \operatorname{if}~~ i=0 \\ \\ \vspace{-1em}
        \begin{bmatrix}
        -d_{i+1}^X & 0 \\
        f_{i+1} & d_i^Y 
        \end{bmatrix} & \operatorname{if}~~ i = 1,2,\dots,n-1 \\ \\ 
        \begin{bmatrix}
        f_{n+1} & d_n^Y 
        \end{bmatrix} & \operatorname{if}~~ i = n.
    \end{cases} \vspace{0.4em}
\]
The \textit{mapping cocone} of a morphism $g_\bullet$ where $g_{n+1}$ is the identity on some object, is defined dually.
\end{definition}

\begin{definition}
An \textit{$n$-exangulated category} is a triplet $(\C,\E,\s)$ of an additive category $\C$, a biadditive functor $\E \colon \C^{\operatorname{op}} \times \C \rightarrow \Ab$ and an exact realization $\s$ of $\E$, satisfying the following axioms:
\begin{enumerate}
    \item [(EA1)]The class of inflations in $\C$ is closed under composition. Dually, the class of deflations in $\C$ is closed under composition.
    \item [(EA2)]For an extension $\delta \in \E(D,A)$ and a morphism $c \in \C(C,D)$, let $\langle X_\bullet, c^*\delta \rangle$ and $\langle Y_\bullet, \delta \rangle$ be distinguished $n$-exangles. Then there exists a \textit{good lift} $f_\bullet$ of $(1_A,c)$, meaning that the mapping cone of $f_\bullet$ gives a distinguished $n$-exangle $\langle M_\bullet^f, (d_0^X)_*\delta \rangle$.
    \item [(EA2)$^{\operatorname{op}}$]Dual of (EA2).
\end{enumerate}
\end{definition}

The condition (EA2) is actually independent of choice of representatives of the equivalence classes $[X_\bullet]$ and $[Y_\bullet]$, see \cite{Herschend-Liu-Nakaoka}*{Corollary 2.31}. Note that we will often not mention $\E$ and $\s$ explicitly when we talk about an $n$-exangulated category $\C$.

Not too surprisingly, a $1$-exangulated category is the same as an extriangulated category \cite{Herschend-Liu-Nakaoka}*{Proposition 4.3}. It should also be noted that $n$-exact and $(n+2)$-angulated categories are $n$-exangulated \cite{Herschend-Liu-Nakaoka}*{Proposition 4.34 and 4.5}. For a discussion of examples of $n$-exangulated categories which are neither $n$-exact nor $(n+2)$-angulated, see \cite{Herschend-Liu-Nakaoka}*{Section 6.3} and \cite{Liu-Zhou}*{Section 4}.

In our study of subcategories of $n$-exangulated categories in \cref{subcategories} and \cref{classification}, the notion of extension-closed subcategories will be relevant. 

\begin{definition}
Let $(\C,\E,\s)$ be an $n$-exangulated category. A full additive subcategory $\S \subseteq \C$ which is closed under isomorphisms is called \textit{extension-closed} if for any pair of objects $A$ and $C$ in $\S$ and any extension $\delta \in \E(C,A)$, there is a distinguished $n$-exangle $\langle X_\bullet,\delta \rangle$ with $X_i$ in $\S$ for $i=1,\dots,n$.
\end{definition}

Extension-closed subcategories inherit structure from the ambient category in a natural way. The following result is \cite{Herschend-Liu-Nakaoka}*{Proposition 2.35}.

\begin{proposition} \label{extension-closed}
Let $(\C,\E,\s)$ be an $n$-exangulated category and $\S$ an extension-closed subcategory of $\C$. Given objects $A$ and $C$ in $\S$ and an extension $\delta \in \E(C,A)$, let $\langle X_\bullet,\delta \rangle$ be a distinguished $n$-exangle with $X_i$ in $\S$ for $i=1,\dots,n$. Define $\mathfrak{t}(\delta) = [X_\bullet]$, where the equivalence class is taken in $\textbf{C}_{(\S;A,C)}^{n+2}$. The following statements hold:
\begin{enumerate}
    \item The correspondence $\mathfrak{t}$ is an exact realization of the restricted functor $\E|_{\S^{\operatorname{op}} \times \S}$, and $(\S,\E|_{\S^{\operatorname{op}} \times \S},\mathfrak{t})$ satisfies (EA2) and (EA2)$^{\operatorname{op}}$.
    \item If $(\S,\E|_{\S^{\operatorname{op}} \times \S},\mathfrak{t})$ satisfies (EA1), then it is an $n$-exangulated category.
\end{enumerate}
\end{proposition}

We end this section by reviewing two results which will be needed throughout the rest of this paper. The following proposition should be well-known, but we include a proof as we lack an explicit reference. The conflations described in \cref{trivial conflations} are called \textit{trivial}.

\begin{proposition} \label{trivial conflations}
Let $\C$ be an $n$-exangulated category and $A$ an object in $\C$. Then the $(n+2)$-term sequence 
\[
0 \rightarrow \cdots \rightarrow 0 \rightarrow A \xrightarrow{1_A} A \rightarrow 0 \rightarrow \cdots \rightarrow 0
\]
which has $A$ in position $i$ and $i+1~$ for some \mbox{$i\in\{0,1,\dots,n\}$} is a conflation.
\end{proposition}

\begin{proof} 
By (R2), the statement is true if $i=0$ or $i=n$. We can hence assume that both our end-terms are zero. Now, our sequence is homotopy equivalent in $\textbf{C}^{n+2}_{(0,0)}$ to 
\[
0 \rightarrow 0 \rightarrow \cdots \rightarrow 0 \rightarrow 0.
\]
As this sequence is a conflation, again by (R2), also the sequence we started with has to be a conflation. 
\end{proof}

As one might expect, the coproduct of two conflations is again a conflation. For a proof of this result, see \cite{Herschend-Liu-Nakaoka}*{Proposition 3.2}.

\begin{proposition} \label{coproduct of conflations}
Let $\C$ be an $n$-exangulated category and $X_\bullet$ and $Y_\bullet$ conflations in $\C$. Then also $X_\bullet \oplus Y_\bullet$ is a conflation. 
\end{proposition}

\section{Subcategories and $n$-(co)generators} \label{subcategories}
In this section we introduce the terminology which is needed in our main result, such as the notion of an $n$-(co)generator, complete subcategories and dense subcategories. We also define $n$-exangulated subcategories, and show that the subcategories which will appear in our classification theorem carry this structure.

\begin{definition}
Let $\C$ be an $n$-exangulated category. A full additive subcategory $\G$ of $\C$ is called an \textit{$n$-generator} (resp. \textit{$n$-cogenerator}) of $\C$ if for each object $A$ in $\C$, there exists a conflation
\begin{align*}
    & A' \rightarrow G_1 \rightarrow \cdots \rightarrow G_n \rightarrow A \\
    (\text{Resp. }~ & A \rightarrow G_1 \rightarrow \cdots \rightarrow G_n \rightarrow A')
\end{align*}
in $\C$ with $G_i$ in $\G$ for $i=1,\dots,n$.
\end{definition}

A $1$-(co)generator is often just called a (co)generator. Our notion of a (co)generator essentially coincides with what is used in \cite{Matsui} and \cite{Zhu-Zhuang}. There, however, it is not assumed that the subcategory $\G$ is additive. Note that it would be possible to prove our results also without this extra assumption, but we have chosen this convention to simplify the statement in \cref{[X]-[G]}. 

We get a trivial example of an $n$-(co)generator by choosing $\G$ to be the entire category $\C$. Another natural example arises if our category has enough projectives or injectives. Let us first recall what this means from \cite{Liu-Zhou}*{Definition 3.2}.

\begin{definition}
Let $\C$ be an $n$-exangulated category. 
\begin{enumerate}
    \item An object $P$ in $\C$ is called \textit{projective} if for any conflation
    \[
    X_0 \xrightarrow{d_0} X_1 \rightarrow \cdots \rightarrow X_n \xrightarrow{d_n} X_{n+1}
    \]
    in $\C$ and any morphism $f \colon P \rightarrow X_{n+1}$, there exists a morphism \mbox{$g \colon P \rightarrow X_n$} such that $d_n \circ g = f$.
    \item The category $\C$ \textit{has enough projectives} if for each object $A$ in $\C$, there exists a conflation 
    \[
    A' \rightarrow P_1 \rightarrow \cdots \rightarrow P_n \rightarrow A
    \]
    in $\C$ with $P_i$ projective for $i=1,\dots,n$.
    \item We define \textit{injective objects} and the notion of having \textit{enough injectives} dually. 
\end{enumerate}
\end{definition}

The notion of having enough projectives or injectives relates well to our definition of an $n$-(co)generator, as demonstrated in the example below. 

\begin{example} \label{n+2-ang. cat. gen.}
Let $\C$ be an $n$-exangulated category. If $\C$ has enough projectives, then the full subcategory $\P \subseteq \C$ of projective objects is an $n$-generator of $\C$. Dually, if $\C$ has enough injectives, the full subcategory $\I\subseteq\C$ of injective objects is an $n$-cogenerator of $\C$. In the case where $\C$ is a Frobenius $n$-exangulated category, as defined in \cite{Liu-Zhou}, the subcategory $\P = \I$ is both an $n$-generator and an $n$-cogenerator of $\C$. 
\end{example}

We will classify subcategories of an $n$-exangulated category which are dense and complete. 

\begin{definition}
Let $\C$ be an $n$-exangulated category and $\S$ a full subcategory of $\C$.
\begin{enumerate}
    \item The subcategory $\S$ is \textit{dense} in $\C$ if each object in $\C$ is a summand of an object in $\S$.
    \item The subcategory $\S$ is \textit{complete} if given any conflation in $\C$ with $n+1$ of its objects in $\S$, also the last object has to be in $\S$.
\end{enumerate}
\end{definition}

Even though it is not a part of the definition, it turns out that given reasonable conditions, complete subcategories are always additive and closed under isomorphisms. 

\begin{lemma} \label{additive subcategory}
Let $\C$ be an $n$-exangulated category. Every complete subcategory $\S$ of $\C$ which contains $0$ is additive and closed under isomorphisms.
\end{lemma}

\begin{proof}
Let $A$ and $B$ be objects in $\S$. By taking the coproduct of two trivial conflations, we get the conflation
\[
A \rightarrow A \oplus B \rightarrow B \rightarrow 0 \rightarrow \cdots \rightarrow 0.
\]
As $0$ is in $\S$, all objects in this sequence except the second one is in $\S$. By completeness, this means that also $A \oplus B$ is in $\S$, which shows additivity.

Given an isomorphism $A \xrightarrow{\simeq} B$ in $\C$, the $(n+2)$-term sequence
\[
A \xrightarrow{\simeq} B \rightarrow 0 \rightarrow \cdots \rightarrow 0
\]
is a conflation in $\C$, as it is equivalent to a trivial conflation. Consequently, if $A$ is in $\S$, then also $B$ has to be there, so $\S$ is closed under isomorphisms. 
\end{proof}

Notice that when a subcategory $\S$ of an $n$-exangulated category is dense, it is automatically non-empty. Whenever $n$ is odd and $\S$ is both dense and complete, our subcategory necessarily contains $0$. This can be seen by taking an object $A$ in $\S$ and using completeness with respect to the conflation
\[
A \xrightarrow{1_A} A \xrightarrow{0} A \xrightarrow{1_A} \cdots \xrightarrow{0} A \xrightarrow{1_A} A \rightarrow 0,
\]
which is a sum of trivial conflations, and in which the last object is the only one not equal to $A$. Consequently, dense and complete subcategories are always additive and isomorphism-closed when $n$ is odd, which will often be the case in our further work. We will show that a stronger statement is true, namely that every such subcategory is actually an \textit{$n$-exangulated subcategory} of the ambient category. The key requirement of an $n$-exangulated subcategory is that the inclusion is an $n$-exangulated functor, as introduced in \cite{BT-Shah}*{Definition 2.31}.

\begin{definition}
Let $(\C_1,\E_1,\s_1)$ and $(\C_2,\E_2,\s_2)$ be $n$-exangulated categories. An additive functor $F \colon \C_1 \longrightarrow \C_2$ is an \textit{$n$-exangulated functor} if there is a natural transformation $\eta \colon \E_1 \longrightarrow \E_2(F^{\operatorname{op}}-, F-)$ such that if $\s_1(\delta) =[X_\bullet]$ for some $\delta \in \E_1(C,A)$, then $\s_2(_{\scriptscriptstyle A}\eta_{\scriptscriptstyle C}(\delta)) = [FX_\bullet]$.
\end{definition}

Notice that the notation $_{\scriptscriptstyle A}\eta_{\scriptscriptstyle C}$ is used for the group homomorphism 
\[
_{\scriptscriptstyle A}\eta_{\scriptscriptstyle C} \colon \E_1(C,A) \longrightarrow \E_2(FC,FA) = \E_2(F^{\operatorname{op}}C,FA)
\]
given by the natural transformation $\eta$. We call $\eta$ an \textit{inclusion} if $_{\scriptscriptstyle A}\eta_{\scriptscriptstyle C}$ is an inclusion of abelian groups for every pair of objects $A$ and $C$.

We are now ready to give the definition of an $n$-exangulated subcategory. 
\begin{definition} \label{def. subcat.}
Let $(\C,\E,\s)$ be an $n$-exangulated category. An \textit{$n$-exangulated subcategory} of $\C$ is a full isomorphism-closed subcategory $\S$ which carries an $n$-exangulated structure $(\S,\E',\s')$ for which the inclusion functor is $n$-exangulated and the associated natural transformation is an inclusion.
\end{definition}

Our definition emphasizes that an $n$-exangulated subcategory \textit{inherits} the structure of the ambient category. In particular, the biadditive functor $\E'$ is an additive subfunctor of the restricted functor $\E|_{\S^{\operatorname{op}} \times \S}$ in the sense of \cite{Herschend-Liu-Nakaoka}*{Definition 3.6}. The exact realizations $\s$ and $\s'$ agree, meaning that if $\s'(\delta) =[X_\bullet]$ for some $\delta \in \E'(C,A) \subseteq \E(C,A)$, then $\s(\delta) = [X_\bullet]$. Notice that the first equivalence class is taken in $\textbf{C}_{(\S;A,C)}^{n+2}$, while the second is taken in $\textbf{C}_{(\C;A,C)}^{n+2}$. In the case $n=1$, the subcategories defined above should be called \textit{extriangulated subcategories}. 

For our applications in \cref{classification}, the most important class of examples of $n$-exangulated subcategories will arise from extension-closed subcategories. In this case we have $\E'=\E|_{\S^{\operatorname{op}} \times \S}$. We also give a basic example where $\E'$ is a proper subfunctor.

\begin{example} \label{example}
(1) Let $\S$ be an extension-closed subcategory of an $n$-exangulated category $(\C,\E,\s)$ and define $\mathfrak{t}$ as explained in \cref{extension-closed}. If the triplet $(\S,\E|_{\S^{\operatorname{op}} \times \S},\mathfrak{t})$ satisfies (EA1), then $\S$ is an $n$-exangulated subcategory of $\C$. Notice that the natural transformation $\eta$ associated to the inclusion functor is given by $_{\scriptscriptstyle A}\eta_{\scriptscriptstyle C} = 1_{\E(C,A)}$ for objects $A$ and $C$ in $\S$.

(2) Let $\C = \Ab$ be the category of abelian groups. This is an extriangulated category with biadditive functor $\E = \operatorname{Ext}_{\C}^1$. Let $\S \subseteq \C$ denote the subcategory of semisimple objects. Using that $\S$ is closed under kernels and cokernels, one can check that $\S$ is an abelian subcategory of $\C$. Consequently, one obtains that $\S$ is an extriangulated subcategory with biadditive functor $\E'=\operatorname{Ext}_{\S}^1$. As $\S$ is not extension-closed in $\C$, we can see that $\E'$ is a proper subfunctor of $\E|_{\S^{\operatorname{op}} \times \S}$.
\end{example}

Let us finish this section by showing that if $n$ is odd, every dense and complete subcategory of an $n$-exangulated category is an $n$-exangulated subcategory.

\begin{proposition} \label{dense + complete = n-exangulated}
Let $(\C,\E,\s)$ be an $n$-exangulated category with $n$ odd and $\S$ a dense and complete subcategory of $\C$. The following statements hold: 
\begin{enumerate}
    \item The subcategory $\S$ is extension-closed. 
    \item The triplet $(\S,\E|_{\S^{\operatorname{op}} \times \S},\mathfrak{t})$, with $\mathfrak{t}$ as defined in \cref{extension-closed}, is an $n$-exangulated subcategory of $\C$.
\end{enumerate}
\end{proposition}

\begin{proof}
As $n$ is odd, it follows from \cref{additive subcategory} that the subcategory $\S$ is additive and isomorphism-closed. 

Let $A$ and $C$ be objects in $\S$ and consider an extension $\delta \in \E(C,A)$. As $\C$ is $n$-exangulated, there is a distinguished $n$-exangle $\langle X_\bullet,\delta \rangle$ in $\C$ with $X_\bullet$ given by
\[
    A \rightarrow X_1 \rightarrow \cdots \rightarrow X_n \rightarrow C.
\]
The objects $X_i$ are not necessarily contained in $\S$, but we will show that we can pick another representative of the equivalence class $[X_\bullet]$ for which this is satisfied. 

For $i=1,\dots,n-1$, use that $\S$ is dense and let $X_i'$ be an object such that $X_i \oplus X_i'$ is in $\S$. By adding trivial conflations involving the objects $X_i$ and $X_i'$ to the conflation above, we get a new conflation
\[
A \rightarrow X_1\oplus X_1' \rightarrow \cdots \rightarrow \bigoplus_{i=1}^{n-1}(X_i \oplus X_i') \rightarrow \widebar{X} \rightarrow C,
\]
where $\widebar{X}= X_1 \oplus X_2' \oplus X_3 \oplus \cdots \oplus X_{n-1}' \oplus X_n$. Notice that each of the trivial conflations we have added are equivalent to the zero conflation, i.e.\ the conflation given by the $(n+2)$-term sequence where every object is zero. Hence, our new conflation represents the same equivalence class as the one we started with.

It remains to observe that every object in our new conflation except possibly $\widebar{X}$ is contained in $\S$. As $\S$ is complete, this means that also $\widebar{X}$ is in $\S$, which proves \textit{(1)}.

For \textit{(2)}, notice that by \cref{extension-closed} and \cref{example} it is enough to verify that (EA1) is satisfied. Let $f$ and $g$ be two composable inflations in $\S$. By the definition of $\mathfrak{t}$, inflations in $\S$ are also inflations in $\C$. As $\C$ satisfies (EA1), there is a conflation
\[
    X_0 \xrightarrow{f \circ g} X_1 \rightarrow \cdots \rightarrow X_n \rightarrow X_{n+1}
\]
in $\C$. By assumption, we know that $X_0$ and $X_1$ are in $\S$, but the same is not necessarily true for the last $n$ objects. However, we apply a similar technique as above to get a conflation where all the objects are in $\S$. For $i = 2,\dots,n$, let $X_i'$ be an object such that $X_i \oplus X_i'$ is in $\S$. Adding trivial conflations to the conflation above yields the conflation
\[
X_0 \xrightarrow{f \circ g} X_1 \rightarrow X_2 \oplus X_2' \rightarrow \cdots \rightarrow \bigoplus_{i=2}^{n}(X_i \oplus X_i') \rightarrow \widebar{X},
\]
where $\widebar{X}$ now denotes the object $X_2 \oplus X_3' \oplus X_4 \oplus \cdots \oplus X_n' \oplus X_{n+1}$. As the first $n+1$ objects in this conflation are in $\S$, so is $\widebar{X}$. Consequently, this is a conflation in $\S$, which shows that $f\circ g$ is an inflation in $\S$. A dual argument shows that the class of deflations in $\S$ is closed under composition.
\end{proof}

\section{The Grothendieck group of an $n$-exangulated category}
Throughout the rest of this paper, we let $\C$ be an essentially small category. Hence, the collection of isomorphism classes $\langle A\rangle$ of objects $A$ in $\C$ forms a set, and we can consider the free abelian group $\F(\C)$ generated by such isomorphism classes. We will define the Grothendieck group of an $n$-exangulated category $\C$ to be a certain quotient of this free abelian group. More precisely, we want to factor out the Euler relations coming from conflations. Given a conflation 
\[
X_\bullet \colon X_0 \rightarrow X_1 \rightarrow \cdots \rightarrow X_n \rightarrow X_{n+1}
\]
in $\C$, the corresponding Euler relation is the alternating sum of isomorphism classes
\[
\chi(X_\bullet)= \langle X_0 \rangle - \langle X_1 \rangle + \cdots + (-1)^{n+1}\langle X_{n+1} \rangle.
\]

\begin{definition}
Let $\C$ be an $n$-exangulated category.
The \textit{Grothendieck group} of $\C$ is the quotient $K_0(\C)=\F(\C)/\R(\C)$, where $\R(\C)$ is the subgroup generated by the subset
\begin{align*}
&\{ \chi(X_\bullet) \mid X_\bullet ~~\text{is a conflation in}~~ \C \} ~~\text{if $n$ is odd and} \\
\{\langle 0 \rangle \}\cup&\{ \chi(X_\bullet) \mid X_\bullet ~~\text{is a conflation in}~~ \C \} ~~\text{if $n$ is even.}
\end{align*}
We denote the equivalence class $\langle A \rangle + \R(\C)$ represented by an object \mbox{$A$ in $\C$} by $[A]$.
\end{definition}

It is immediate from the definition that the Grothendieck group $K_0(\C)$ has a universal property. Namely, any homomorphism of abelian groups from $\F(\C)$ satisfying the Euler relations factors uniquely through $K_0(\C)$. More precisely, given any abelian group $T$ and a homomorphism $t \colon \F(\C) \rightarrow T$ with $t(\R(\C))=0$, there exists a unique homomorphism $t'$ such that the following diagram commutes
\[
\begin{tikzcd}[column sep=11]
\F(\C) \arrow[r,"\pi"] \arrow[rd,"t"] & K_0(\C) \arrow[d,dotted,"t'"] \\
& T,
\end{tikzcd}
\]
where $\pi$ is the natural projection.

Let us prove some basic properties of the Grothendieck group of an $n$-exangulated category. These properties are well-known in the cases where our category is triangulated or exact. Note that $\langle 0 \rangle$ was defined to be in $\R(\C)$ whenever $n$ is even in order for the following proposition to hold.

\begin{proposition} \label{grunnleggende K_0}
Let $\C$ be an $n$-exangulated category.
\begin{enumerate}
    \item The zero element in $K_0(\C)$ is given by $[0]$, where $0$ is the zero object in $\C$.
    \item For objects $A$ and $B$ in $\C$, we have $[A\oplus B]=[A]+[B]$ in $K_0(\C)$.
\end{enumerate}
\end{proposition}

\begin{proof}
If $n$ is even, the definition of $\R(\C)$ immediately implies that $[0]$ is the zero element in $K_0(\C)$. 

Recall that the $(n+2)$-term sequence 
\[
0 \rightarrow 0 \rightarrow \cdots \rightarrow 0 \rightarrow 0
\]
is a conflation in $\C$ by (R2). Consequently, the sum $\sum_{i=0}^{n+1}(-1)^i \langle 0 \rangle$ is in $\R(\C)$. If $n$ is odd, this sum is equal to $\langle 0 \rangle$, and hence $[0]$ is the zero element in $K_0(\C)$ also in this case. This shows \textit{(1)}.

For \textit{(2)}, consider the sequence 
\[
A \rightarrow A \oplus B \rightarrow B \rightarrow 0 \rightarrow \cdots \rightarrow 0
\]
with $n+2$ terms. This sequence is a conflation in $\C$ as it is a sum of two trivial conflations. Using \textit{(1)}, this implies that
\[
\langle A \rangle - \langle A \oplus B \rangle + \langle B \rangle \in \R(\C),
\]
which yields $[A \oplus B] = [A] + [B]$ in $K_0(\C)$.
\end{proof}

Notice that any element in $K_0(\C)$ can be written as $[A]-[B]$ for some objects $A$ and $B$ in $\C$, as we can collect positive and negative terms and then use the second part of the proposition above. In the case where $n$ is odd and our category has an $n$-(co)generator, we get an even nicer description.

\begin{proposition} \label{[X]-[G]}
Let $\C$ be an $n$-exangulated category with $n$ odd. Let $\G$ be an $n$-(co)generator of $\C$. Then every element in $K_0(\C)$ can be written as $[A]-[G]$ for some objects $A$ in $\C$ and $G$ in $\G$.
\end{proposition}

\begin{proof}
Given an element in $K_0(\C)$, we know that it can be written as \mbox{$[X]-[B]$} for some objects $X$ and $B$ in $\C$. When $\G$ is an $n$-generator, there exists a conflation
\[
B' \rightarrow G_1 \rightarrow \cdots \rightarrow G_n \rightarrow B
\]
in $\C$ with $G_i$ in $\G$ for $i=1,\dots,n$. Consequently, using that $n$ is odd, we get
\begin{equation} \tag{$\ast$}
    [B] = -[B'] + [G_1] - [G_2] + \cdots - [G_{n-1}] +[G_n]
\end{equation}
in $K_0(\C)$. Substituting this expression for $[B]$, the element we started with can be written as
\begin{align*}
[X] - [B] &= [X] + [B'] - [G_1] + [G_2] - \cdots + [G_{n-1}] - [G_n] \\
&= [X \oplus B' \oplus G_2 \oplus G_4 \oplus \cdots \oplus G_{n-1}] - [G_1 \oplus G_3 \oplus \cdots \oplus G_n],           
\end{align*}
where we have collected positive and negative terms from the alternating sum and used \cref{grunnleggende K_0}. Defining $A$ and $G$ to be the objects in the first and second bracket respectively, we get that our element can be written as $[A]-[G]$. Note that as $\G$ is additive, the object $G$ is contained in $\G$.

The proof in the case where $\G$ is an $n$-cogenerator is dual.
\end{proof}

Note that it was important in the argument above that $n$ was assumed to be odd. If $n$ was even, there would be no negative sign in front of the term $[B']$ in the expression ($\ast$). Hence, the signs of $[X]$ and $[B']$ in our final equation would be different, and we would not reach our conclusion.

The description of elements in the Grothendieck group which is provided in \cref{[X]-[G]} will be important in our further work. In the following, we will thus often need to assume that $n$ is odd.

\begin{remark}
\cref{[X]-[G]} is an $n$-exangulated analogue of a result from \cite{Matsui} concerning exact categories, which can be found in the proof of Lemma 2.8. As an $(n+2)$-angulated category has $\G=\{0\}$ as an $n$-(co)generator, \cref{[X]-[G]} can also be thought of as a generalization of part \textit{(3)} of \cite{Bergh-Thaule}*{Proposition 2.2}.
\end{remark}

\section{Classification of subcategories}
\label{classification}

Recall that $\C$ is assumed to be essentially small. In this section we state and prove our main result. For $n$ odd we classify dense complete subcategories of an $n$-exangulated category with an $n$-(co)generator $\G$ in terms of subgroups of the Grothendieck group. The subgroups which appear in the bijection, depend on the $n$-(co)generator. More precisely, the subgroups have to contain 
\[
H_{\G} = \langle [G] \in K_0(\C) \mid G \in \G \rangle \leq K_0(\C),
\]
i.e.\ the subgroup of $K_0(\C)$ generated by elements represented by objects \mbox{in $\G$}. When a subgroup of $K_0(\C)$ contains $H_\G$, we say that it contains the image of $\G$.

\begin{theorem} \label{main theorem}
Let $\C$ be an $n$-exangulated category with $n$ odd. Let $\G$ be an $n$-(co)generator of $\C$. There is then a one-to-one correspondence
\[
\left\{
\begin{tabular}{@{}l@{}}
    subgroups of $K_0(\C)$ \\
    containing $H_{\G}$ 
\end{tabular}
\right\}
\xrightleftarrows{f}{g}
\left\{
\begin{tabular}{@{}l@{}}
    dense complete subcategories \\
    of $\C$ containing $\G$ 
\end{tabular}
\right\},
\]
where $f(H)$ is the full subcategory
\[
f(H) = \{A \in \C \mid [A] \in H\} \subseteq \C, 
\]
and $g(\S)$ is the subgroup
\[
     g(\S) = \langle [A] \in K_0(\C) \mid A \in \S \rangle \leq K_0(\C).
\]
\end{theorem}

\begin{remark}
The subcategories in our bijection above are $n$-exangulated subcategories of $\C$, where the $n$-exangulated structure is inherited from that of $\C$ as described in \cref{dense + complete = n-exangulated}.
\end{remark}

\begin{proof}[Proof of Theorem 5.1]
We prove the theorem in the case where $\G$ is an $n$-generator. The proof when $\G$ is an $n$-cogenerator is dual. 

Throughout the rest of the proof, let $\S$ be a dense complete subcategory of $\C$ containing $\G$ and $H$ a subgroup of $K_0(\C)$ containing $H_\G$. Let us first verify that the maps $f$ and $g$ actually end up where we claim. 

Note that $g(\S)$ is a subgroup of $K_0(\C)$ by definition. As $\S$ contains $\G$, the subgroup $H_\G$ is contained in $g(\S)$. Similarly, it is clear that $\G \subseteq f(H)$. To see that $f(H)$ is a dense subcategory, let $A$ be an object in $\C$. As $\G$ is an $n$-generator, there is a conflation
\[
A' \rightarrow G_1 \rightarrow \cdots \rightarrow G_n \rightarrow A
\]
in $\C$ with $G_i$ in $\G$ for $i=1,\dots,n$. Using that $n$ is odd, which implies that the signs in front of $[A]$ and $[A']$ in the corresponding Euler relation agree, we get 
\[
[A \oplus A'] = [G_1] - [G_2] + \cdots - [G_{n-1}] + [G_n] \in H.
\] 
This means that $A \oplus A'$ is in $f(H)$, so the subcategory is dense in $\C$. To show completeness, consider a conflation
\[
X_0 \rightarrow X_1 \rightarrow \cdots \rightarrow X_n \rightarrow X_{n+1}
\]
in $\C$, where $n+1$ of the $n+2$ objects are in $f(H)$. Since 
\[
[X_0] - [X_1] + \cdots + (-1)^{n+1}[X_{n+1}] = 0 \in H,
\]
and $n+1$ of the terms in this sum are in $H$, also the last term has to be there. This means that also the last object of our conflation above is in $f(H)$, so $f(H)$ is complete.

Our next step is to show that $f$ and $g$ are inverse bijections. The inclusion $gf(H)\subseteq H$ is immediate. For the reverse inclusion, choose an element in $H$. By \cref{[X]-[G]}, our element can be written as $[A]-[G]$ for some $A$ in $\C$ and $G$ in $\G$. As $[A] = ([A] - [G]) + [G]$, and both $[A]-[G]$ and $[G]$ are in $H$, so is $[A]$. Hence, our element is contained in $gf(H)$, and we can conclude that $H=gf(H)$. 

It remains to show that $\S = fg(\S)$. Again, one of the inclusions is clear from the definitions, namely $\S \subseteq fg(\S)$. For the reverse inclusion, choose an object $A$ in $fg(\S)$. This means that $[A]$ is in $g(\S)$. By \cref{necessary lemma} below, our object $A$ is consequently in $\S$, which completes our proof.
\end{proof}

We will prove \cref{necessary lemma} by showing that the quotient $K_0(\C) / g(\S)$ is isomorphic to another group $G_\S$ consisting of equivalence classes. 

Given an $n$-exangulated category $\C$ with $n$ odd and a dense complete subcategory $\S$ of $\C$, define a relation $\sim$ on the set of isomorphism classes of objects in $\C$ by $\langle A \rangle \sim \langle B \rangle$ if and only if there exist objects $S_A$ and $S_B$ in $\S$ such that $A \oplus S_A \simeq B \oplus S_B$. One can check that this is an equivalence relation. Denote by $G_\S$ the quotient of the set of isomorphism classes of objects in $\C$ by the relation $\sim$. Elements in $G_\S$ are denoted by $\{A\}$.

\begin{lemma} \label{first lemma}
Let $\C$ be an $n$-exangulated category with $n$ odd and $\S$ a dense complete subcategory of $\C$. An object $A$ in $\C$ is contained in $\S$ if and only if $\{A\} = \{0\}$ in $G_\S$.
\end{lemma}

\begin{proof}
If $A$ is in $\S$, then clearly $\{A\}=\{0\}$. Conversely, assume $\{A\}=\{0\}$. This means that there are objects $S_A$ and $S_0$ in $\S$ such that $A \oplus S_A \simeq S_0$. Consequently, the $n+1$ last objects in the $(n+2)$-term sequence
\[
A \rightarrow A \oplus S_A \rightarrow S_A \rightarrow 0 \rightarrow \cdots \rightarrow 0
\]
are in $\S$. This is a conflation as it is the coproduct of two trivial conflations. Hence, as $\S$ is complete, our object $A$ is also in $\S$.
\end{proof}

\begin{lemma} \label{necessary lemma}
Let $\C$ be an $n$-exangulated category with $n$ odd. Let $\G$ be an $n$-(co)generator of $\C$ and $\S$ a dense complete subcategory of $\C$ which contains $\G$. The following statements hold:
\begin{enumerate}
    \item $G_\S$ is an abelian group with binary operation $\{A\} + \{B\} \coloneqq \{A \oplus B\}$ and identity element $\{0\}$.
    \item The map
    \begin{align*}
    K_0(\C) / g(\S)
    &\xrightarrow{\mathmakebox[2em]{\simeq}} G_\S \\
    [A] + g(\S) &\longmapsto \{A\}
    \end{align*}
    is a well-defined isomorphism of groups. In particular, an object $A$ in $\C$ is contained in $\S$ if and only if $[A]$ is in $g(\S)$. 
\end{enumerate}
\end{lemma}

\begin{proof}
In order to show \textit{(1)}, notice first that our binary operation is well-defined, commutative, associative and has $\{0\}$ as identity element. For any object $A$ in $\C$, there exists an object $A'$ such that $A \oplus A'$ is in $\S$, by denseness of $\S$. Using \cref{first lemma}, this means that 
\[
\{A\} + \{A'\} = \{A \oplus A'\} = \{0\}.
\]
Hence, the element $\{A'\}$ is the inverse of $\{A\}$, and $G_\S$ is an abelian group.

For \textit{(2)}, let us first show that the map
\begin{align*}
    \phi \colon K_0(\C)
    &\longrightarrow G_\S \\
    [A] &\longmapsto \{A\}
\end{align*}
is well-defined. It suffices to show that the Euler relations are sent to zero. Consider a conflation
\[
X_0 \rightarrow X_1 \rightarrow \cdots \rightarrow X_n \rightarrow X_{n+1}
\]
in $\C$. For $i=1,\dots,n+1$, let $X_i'$ be an object such that $X_i \oplus X_i'$ belongs to $\S$. We can get a new conflation by adding trivial conflations involving the objects $X_i$ and $X_i'$ to the conflation above, namely
\[
\widebar{X} \rightarrow \bigoplus_{i=1}^{n+1}(X_i \oplus X_i') \rightarrow \cdots \rightarrow \bigoplus_{i=n}^{n+1}(X_i \oplus X_i') \rightarrow X_{n+1} \oplus X_{n+1}',
\]
where $\widebar{X}=X_0 \oplus X_1' \oplus X_2 \oplus \cdots \oplus X_n' \oplus X_{n+1}$. As the $n+1$ last objects in this conflation are in $\S$, so is $\widebar{X}$. Consequently, using \cref{first lemma}, we have
\begin{align*}
\{0\} = \{\widebar{X}\} &= \{X_0\} + \{X_1'\} + \{X_2\} + \cdots + \{X_n'\} + \{X_{n+1}\} \\
    &= \{X_0\} - \{X_1\} + \{X_2\} + \cdots - \{X_n\} + \{X_{n+1}\}
\end{align*}
in $G_\S$, so $\phi$ is well-defined. It is now easy to check that $\phi$ is a surjective group homomorphism.

Our last step is to show that $\Ker\phi=g(\S)$. Note that the inclusion $g(\S) \subseteq \Ker\phi$ follows immediately by \cref{first lemma}. Using \cref{[X]-[G]}, any element in $\Ker\phi$ can be written as $[A]-[G]$ for some objects $A$ in $\C$ and $G$ in $\G$. This means that
\[
\{0\} = \phi([A]-[G]) = \{A\}-\{G\}=\{A\},
\]
where the third equality follows from \cref{first lemma} and the assumption that $\S$ contains $\G$. Consequently, again using \cref{first lemma}, the object $A$ is in $\S$. This yields our reverse inclusion. Combining the isomorphism $K_0(\C)/g(\S) \simeq G_S$ and \cref{first lemma}, we see that an object $A$ is in $\S$ if and only if $[A]$ is in $g(\S)$.
\end{proof}

Our main theorem, \cref{main theorem}, extends and unifies results by Thomason, Bergh--Thaule, Matsui and Zhu--Zhuang. We also get a classification of subcategories of $n$-exact categories.

\begin{corollary} \label{corollary}
\begin{enumerate}
    \item \cite{Thomason}*{Theorem 2.1} Let $\C$ be a triangulated category. Then there is a one-to-one correspondence between the dense triangulated subcategories of $\C$ and the subgroups of $K_0(\C)$.
    \item \cite{Bergh-Thaule}*{Theorem 4.6} Let $\C$ be an $(n+2)$-angulated category with $n$ odd. Then there is a one-to-one correspondence between the dense complete $(n+2)$-angulated subcategories of $\C$ and the subgroups of $K_0(\C)$.
    \item \cite{Matsui}*{Theorem 2.7} Let $\C$ be an exact category with a (co)generator $\G$. Then there is a one-to-one correspondence between the dense $\G$-(co)resolving subcategories of $\C$ and the subgroups of $K_0(\C)$ containing the image of $\G$.
    \item \cite{Zhu-Zhuang}*{Theorem 5.7} Let $\C$ be an extriangulated category with a (co)generator $\G$. Then there is a one-to-one correspondence between the dense $\G$-(co)resolving subcategories of $\C$ and the subgroups of $K_0(\C)$ containing the image of $\G$.
    \item Let $\C$ be an $n$-exact category with $n$ odd. Let $\G$ be an $n$-(co)generator of $\C$. Then there is a one-to-one correspondence between the dense complete subcategories of $\C$ containing $\G$ and the subgroups of $K_0(\C)$ containing the image of $\G$.
\end{enumerate}
\end{corollary}

\begin{proof}
Part \textit{(5)} follows immediately from \cref{main theorem} as $n$-exact categories are $n$-exangulated. 

As \textit{(2)} implies \textit{(1)} and \textit{(4)} implies \textit{(3)}, it suffices to prove \textit{(2)} and \textit{(4)}. To show \textit{(4)}, notice that in the case $n=1$, a dense subcategory containing $\G$ is complete if and only if it is $\G$-(co)resolving as defined in \cite{Zhu-Zhuang}*{Definition 5.3}. It is thus clear that \cref{main theorem} implies \textit{(4)}.

In order to prove that \textit{(2)} follows from our main result, we will use the definition of an $(n+2)$-angulated category, see \cites{Geiss-Keller-Oppermann, Bergh-Thaule 2013}. 

Let $(\C,\Sigma)$ be an $(n+2)$-angulated category. Then $\C$ has enough projectives, with $0$ as the only projective object. The same is true for injectives, and hence $\G = \{0\}$ is both an $n$-generator and an $n$-cogenerator of $\C$. As a subgroup necessarily contains the zero element, all subgroups of $K_0(\C)$ will contain the image of $\G$.

It remains to show that every complete and dense subcategory $\S$ of $\C$ has a natural structure as an $(n+2)$-angulated subcategory, by declaring the distinguished $(n+2)$-angles in $\C$ with all objects in $\S$ to be the distinguished $(n+2)$-angles in $\S$. Recall that a full isomorphism-closed subcategory $\S$ of our $(n+2)$-angulated category $(\C,\Sigma)$ is an $(n+2)$-angulated subcategory if $(\S,\Sigma)$ itself is $(n+2)$-angulated and the inclusion is an $(n+2)$-angulated functor. 

Recall from \cref{subcategories} that as $n$ is odd, the subcategory $\S$ contains $0$, which again implies that it is additive and isomorphism-closed. To show that $(\S,\Sigma)$ is $(n+2)$-angulated, the crucial parts are to check that $\S$ is closed under $\Sigma$ and that morphisms can be completed to distinguished $(n+2)$-angles. 

When we think of $\C$ as an $n$-exangulated category, the distinguished $(n+2)$-angles yield conflations when we remove the last object. To see that $\S$ is closed under $\Sigma$, let $A$ be an object in $\S$. As 
\[
A \rightarrow 0 \rightarrow \cdots \rightarrow 0 \rightarrow \Sigma A \xrightarrow{1_{\Sigma A}} \Sigma A
\]
is a distinguished $(n+2)$-angle in which the $n+1$ first objects are in $\S$, also $\Sigma A$ is in $\S$. A dual argument shows that $\Sigma^{-1} A$ is in $\S$.

Let $f \colon X_0 \rightarrow X_1$ be a morphism in $\S$. We need to show that $f$ can be completed to a distinguished $(n+2)$-angle in $\S$. As $\C$ is $(n+2)$-angulated, there is a distinguished $(n+2)$-angle
\[
X_0 \xrightarrow{f} X_1 \rightarrow X_2 \rightarrow \cdots \rightarrow X_{n+1} \rightarrow \Sigma X_0
\]
in $\C$. For $i=2,\dots,n$, use that $\S$ is dense and let $X_i'$ be an object such that $X_i \oplus X_i'$ is in $\S$. By adding trivial $(n+2)$-angles involving the objects $X_i$ and $X_i'$ to the $(n+2)$-angle above, we get a new distinguished $(n+2)$-angle
\[
X_0 \xrightarrow{f} X_1 \rightarrow X_2 \oplus X_2' \rightarrow \cdots \rightarrow \bigoplus_{i=2}^{n}(X_i \oplus X_i') \rightarrow \widebar{X} \rightarrow \Sigma X_0,
\]
where $\widebar{X}=X_2 \oplus X_3' \oplus X_4 \oplus \cdots \oplus X_n' \oplus X_{n+1}$. As the $n+1$ first objects in this sequence are contained in $\S$, so is $\widebar{X}$. Consequently, this is a distinguished $(n+2)$-angle in $\S$ which completes the morphism $f$. 

The remaining axioms of an $(n+2)$-angulated category are immediately verified using the fact that $\S$ is full. As the distinguished $(n+2)$-angles in $\S$ are chosen in such a way that the inclusion functor is $(n+2)$-angulated, we can conclude that $\S$ is an $(n+2)$-angulated subcategory of $\C$. Consequently, also \textit{(2)} follows from \cref{main theorem}.
\end{proof}

\begin{acknowledgements}
The author would like to thank her supervisor Petter Andreas Bergh for helpful discussions and comments. She would also thank Louis-Philippe Thibault for careful reading and helpful suggestions on a previous version of this paper.
\end{acknowledgements}


\begin{bibdiv}
\begin{biblist}

\bib{BT-Shah}{article}{
   author={Bennett-Tennenhaus, Raphael},
   author={Shah, Amit},
   title={Transport of structure in higher homological algebra},
   journal={arXiv:2003.02254v2},
   date={2020},
}

\bib{Bergh-Thaule 2013}{article}{
   author={Bergh, Petter Andreas},
   author={Thaule, Marius},
   title={The axioms for $n$-angulated categories},
   journal={Algebr. Geom. Topol.},
   volume={13},
   date={2013},
   number={4},
   pages={2405--2428},
}

\bib{Bergh-Thaule}{article}{
   author={Bergh, Petter Andreas},
   author={Thaule, Marius},
   title={The Grothendieck group of an $n$-angulated category},
   journal={J. Pure Appl. Algebra},
   volume={218},
   date={2014},
   number={2},
   pages={354--366},
}

\bib{Geiss-Keller-Oppermann}{article}{
   author={Geiss, Christof},
   author={Keller, Bernhard},
   author={Oppermann, Steffen},
   title={$n$-angulated categories},
   journal={J. Reine Angew. Math.},
   volume={675},
   date={2013},
   pages={101--120},
}

\bib{Herschend-Liu-Nakaoka}{article}{
   author={Herschend, Martin},
   author={Liu, Yu},
   author={Nakaoka, Hiroyuki},
   title={$n$-exangulated categories},
   journal={arXiv:1709.06689v3},
   date={2017},
}

\bib{Iyama}{article}{
   author={Iyama, Osamu},
   title={Cluster tilting for higher Auslander algebras},
   journal={Adv. Math.},
   volume={226},
   date={2011},
   number={1},
   pages={1--61},
}
 
\bib{Iyama-Nakaoka-Palu}{article}{
   author={Iyama, Osamu},
   author={Hiroyuki, Nakaoka},
   author={Yann, Palu},
   title={Auslander--Reiten theory in extriangultaed categories},
   journal={arXiv:1805.03776},
   date={2018},
}

\bib{Iyama-Oppermann}{article}{
   author={Iyama, Osamu},
   author={Oppermann, Steffen},
   title={$n$-representation-finite algebras and $n$-APR tilting},
   journal={Trans. Amer. Math. Soc.},
   volume={363},
   date={2011},
   number={12},
   pages={6575--6614},
}

\bib{Jasso}{article}{
   author={Jasso, Gustavo},
   title={$n$-abelian and $n$-exact categories},
   journal={Math. Z.},
   volume={283},
   date={2016},
   number={3-4},
   pages={703--759},
}
	
\bib{Landsburg}{article}{
   author={Landsburg, Steven E.},
   title={$K$-theory and patching for categories of complexes},
   journal={Duke Math. J.},
   volume={62},
   date={1991},
   number={2},
   pages={359--384},
}

\bib{Liu-Zhou}{article}{
   author={Liu, Yu},
   author={Zhou, Panyue},
   title={Frobenius $n$-exangulated categories},
   journal={arXiv:1909.13284},
   date={2019},
}

\bib{Matsui}{article}{
   author={Matsui, Hiroki},
   title={Classifying dense resolving and coresolving subcategories of exact
   categories via Grothendieck groups},
   journal={Algebr. Represent. Theory},
   volume={21},
   date={2018},
   number={3},
   pages={551--563},
}

\bib{Nakaoka-Palu}{article}{
   author={Nakaoka, Hiroyuki},
   author={Palu, Yann},
   title={Extriangulated categories, Hovey twin cotorsion pairs and model
   structures},
   language={English, with English and French summaries},
   journal={Cah. Topol. G\'{e}om. Diff\'{e}r. Cat\'{e}g.},
   volume={60},
   date={2019},
   number={2},
   pages={117--193},
}

\bib{Neeman 1991}{article}{
   author={Neeman, Amnon},
   title={Some new axioms for triangulated categories},
   journal={J. Algebra},
   volume={139},
   date={1991},
   number={1},
   pages={221--255},
}

\bib{Neeman 2001}{book}{
   author={Neeman, Amnon},
   title={Triangulated categories},
   series={Annals of Mathematics Studies},
   volume={148},
   publisher={Princeton University Press, Princeton, NJ},
   date={2001},
   pages={viii+449},
}

\bib{Thomason}{article}{
   author={Thomason, Robert Wayne},
   title={The classification of triangulated subcategories},
   journal={Compositio Math.},
   volume={105},
   date={1997},
   number={1},
   pages={1--27},
}

\bib{Zhu-Zhuang}{article}{
   author={Zhu, Bin},
   author={Zhuang, Xiao},
   title={Grothendieck groups in extriangulated categories},
   journal={arXiv:1912.00621v4},
   date={2019},
}

\end{biblist}
\end{bibdiv}
\end{document}